\documentclass[12pt, reqno]{amsproc}

\usepackage[margin = 1.1in]{geometry}
\usepackage{amsmath,amsfonts,amssymb,amsthm}
\usepackage[usenames,dvipsnames]{color}
\usepackage[all]{xy}
\usepackage{pb-diagram,pb-xy}
\usepackage{mathpazo,mathrsfs}
\usepackage{eulervm}
\usepackage{nicefrac}
\usepackage{graphicx}

\usepackage{hyperref}
\hypersetup{
    colorlinks,	%
    citecolor=blue,%
    filecolor=black,%
    linkcolor=red,%
    urlcolor=gray
}

\numberwithin{equation}{section}

\theoremstyle{plain}
	\newtheorem{theorem}[equation]{Theorem}
	\newtheorem{lemma}[equation]{Lemma}
	
	\newtheorem{proposition}[equation]{Proposition}

	\newtheorem*{theorem*}{Theorem}

\theoremstyle{definition}
	\newtheorem{definition}[equation]{Definition}
	\newtheorem{example}[equation]{Example}

\theoremstyle{remark}
	
	\newtheorem*{remark*}{Remark}

 { \begin{list}%
         {$\bullet$}%
         {\setlength{\labelwidth}{20pt}%
          \setlength{\leftmargin}{25pt}%
          \setlength{\topsep}{0pt}
          \setlength{\itemsep}{1.5ex}
          \setlength{\parsep}{0pt}}}%
 { \end{list} }

\newcommand{\R}{\mathbb{R}}

\newcommand{\U}{{U}}
\newcommand{\V}{{V}}

\newcommand{\ob}{_{o}}

\newcommand{\bR}{\mathbf{R}}
\newcommand{\bC}{\mathbf{C}}
\newcommand{\bD}{\mathbf{D}}

\newcommand{\Rcat}{\mathbf{R}}
\newcommand{\Vect}{\mathrm{\bf Vect}}
\newcommand{\Mod}{\mathrm{\bf Mod}}

\newcommand{\C}{\mathrm{\bf C}}
\newcommand{\RC}{\bC^\bR}
\newcommand{\RD}{\bD^\bR}

\newcommand{\Met}{\mathrm{\bf Met}}
\newcommand{\Cat}{\mathrm{\bf Cat}}

\newcommand{\Top}{\mathrm{\bf Top}}
\newcommand{\Zcat}{\mathbf{Z}}
\newcommand{\Ncat}{\mathbf{N}}
\newcommand{\ncat}{\mathbf{n}}

\newcommand{\id}{\mathrm{I}}

\newcommand{\To}{\Rightarrow}
\newcommand{\xto}{\xrightarrow}

\newcommand{\intdist}{\mathrm{d}_\mathrm{Int}}
\newcommand{\botdist}{\mathrm{d}_\mathrm{Bot}}

\newcommand{\dgm}{\mathrm{dgm}\vspace{0.25em}}

\newcommand\set[1]{\left\{#1\right\}}

\newcommand{\Rips}{\mathcal{V}}
\newcommand{\Cech}{\mathcal{C}}

\begin{document}

\title[Higher Interpolation and Extension for Persistence Modules]{Higher Interpolation and Extension\\ for Persistence Modules}

\author{Peter Bubenik}
\address[Peter Bubenik]{Department of Mathematics, University of Florida}
\email{peter.bubenik@ufl.edu}

\author{Vin de Silva}
\address[Vin de Silva]{Department of Mathematics, Pomona College}
\email{Vin.deSilva@pomona.edu}

\author{Vidit Nanda}
\address[Vidit Nanda]{{Mathematics Institute, University of Oxford and The Alan Turing Institute}}
\email{nanda@maths.ox.ac.uk}

\begin{abstract}
The use of topological persistence in contemporary data analysis has provided considerable impetus for investigations into the geometric and functional-analytic structure of the space of persistence modules. In this paper, we isolate a coherence criterion which guarantees the extensibility of non-expansive maps into this space across embeddings of the domain to larger ambient metric spaces. Our coherence criterion is category-theoretic, allowing Kan extensions to provide the desired extensions. 
Our main construction gives an isometric embedding of a metric space into the metric space of persistence modules with values in the spacetime of this metric space.
As a consequence of such ``higher interpolation'', it becomes possible to compare Vietoris-Rips and \v{C}ech complexes built within the space of persistence modules.

\end{abstract}

\maketitle

\section*{Introduction}

The combination of rigorously-developed foundations \cite{bs,cdsgo}, efficient computability \cite{zc,mors} and stability properties \cite{cseh, ccsgo} have resulted in  the widespread adoption of {\bf topological persistence} \cite{elz, zc} as a technique for the analysis of large and complex datasets \cite{car, ghr, ns}. The output of this process is a collection of {\em persistent homology groups}, which are typically represented via a barcode or a persistence diagram. Recent applications of persistence often confront dynamically evolving data \cite{kolm, ricc}, and in these cases one requires the ability to make inferences about the dynamics from collections of persistence diagrams. Substantial efforts have been devoted to this end; among the best-known outcomes are {\em vineyards} \cite{vineyards}, {\em Fr\'echet means} \cite{frechet} and {\em persistence landscapes} \cite{landscapes}.

In this work, we provide a new geometric lens with which to view the space of persistence diagrams. Our main result is in fact a statement about the space of (sufficiently tame) {\em persistence modules} --- these consist of vector spaces and linear maps indexed by the real line $\R$, and their representation theory produces persistence diagrams \cite{oudot}. The class $\Mod$ of persistence modules admits an {\em interleaving metric}, and the {\em interpolation lemma} from \cite{ccsgo} establishes that $\Mod$ is a path metric space --- two modules which are $e$-interleaved for $0 \leq e < \infty$ can always be connected by a path in $\Mod$ of length $e$. This lemma plays a key role in the proof of the {\bf stability theorem}, which confirms that if two point-clouds are within Hausdorff distance $e$ of each other then (their persistence modules are $e$-interleaved, and hence) their persistence diagrams are also within {\em bottleneck distance} $e$ of each other \cite{cseh, ccsgo}.

The interpolation lemma provides an affirmative answer to the {\em Lipschitz extension problem} \cite[Ch 1]{blin} encoded in the following commutative diagram of metric spaces (and $1$-Lipschitz maps):
\begin{equation} \label{eq:ext1}
\xymatrixrowsep{0.4in}
\xymatrixcolsep{0.4in}
\xymatrix{
\{0,e\} \ar@{->}[r]^{f} \ar@{_{(}->}[d] & \Mod \\
[0,e] \ar@{-->}[ur]_{f'} & 
}
\end{equation}
Here $\{0,e\}$ and $[0,e]$ are given the traditional Euclidean metric inherited from $\R$, and the fact that $f$ is $1$-Lipschitz follows immediately from our assumption that $f(0)$ and $f(e)$ are $e$-interleaved. The existence of an extension $f'$ allows us to assign intermediate persistence modules $f'(x)$ to all $x$ in $[0,e]$ so that $f'$ agrees with $f$ on the endpoints $\{0,e\}$ and the interleaving distance between $f'(x)$ and $f'(y)$ does not exceed $|x-y|$. Similarly, one seeks $1$-Lipschitz extensions across more general choices of metric inclusions. Our objective here is to prescribe sufficient categorical conditions on $f$ which guarantee the existence of such extensions. Here is a consequence of our main result:
\begin{theorem*}[\textbf{Higher interpolation and extension}]
Let $M$ be any metric space and $A$ a subspace. If a map $f:A \to \Mod$ is {\bf coherent} (in the sense of Definition \ref{def:coh} below), then it admits three $1$-Lipschitz extensions $M \to \Mod$.
\end{theorem*}

In order to precisely describe what it takes for $f:A \to \Mod$ to be coherent, we examine a pair of functors relating $\Cat$, the usual category of small categories, and $\Met$, the category of metric spaces with $1$-Lipschitz maps as morphisms. Although our functors fail to form an adjoint pair in general, there is a distinguished natural transformation $\eta$ from the identity functor on $\Met$ to their composite. Coherent maps are precisely those $f: A \to \Mod$ which factor through this natural transformation. The rest of this paper is organized as follows: in Section \ref{sec:geom} we use known facts about the metric space of persistence modules to describe a functor $\Cat \to \Met$, and in Section \ref{sec:met-cat} we describe a functor $\Met \to \Cat$. The proof of the higher interpolation theorem occupies Section \ref{sec:coh} and some of its consequences are explored in Section \ref{sec:cons}.

%

\section{The Geometry of Persistence Modules}\label{sec:geom}

We assume that the reader has prior familiarity with the basics of category theory \cite{maclane,ahs:book}. We also adopt the following conventions throughout: given a category $\bC$ we write $\bC\ob$ for its class of objects and $\bC(x,y)$ for its set of morphisms from an object $x$ to an object $y$. For a small category $\bC$, we will denote the {\em category of functors} from $\bC$ to $\bD$ by $\bD^\bC$. Although we will survey some relevant definitions and results here, the reader is invited to consult \cite{bs, bdss,ccsgo,cdsgo,oudot} for detailed background material on the categorical and metric aspects of persistence modules.

\subsection{Persistence Modules as Functors}

Let $\Rcat$ denote the category whose objects are the real numbers $\R$, and which admits a unique morphism $a \to b$ whenever $a \leq b$. Persistent homology associates algebraic invariants to {\em filtered topological spaces}, which are naturally regarded as members of $\Top^\Rcat$ -- these are functors from $\bR$ to the category $\Top$ of topological spaces and continuous maps. In practice, one also encounters filtered spaces indexed by proper subcategories of $\Rcat$ (typically finite sets $\ncat = \{0,1,\ldots,n\}$, natural numbers $\Ncat$, or integers $\Zcat$). In such cases, a standard dictionary may be used to modulate between indexing subcategories: in the diagram
\begin{equation} \label{eq:nzr}
\xymatrixrowsep{0.5in}
\xymatrixcolsep{0.3in}
  \xymatrix{\ncat \ar[r] \ar[drrr] & \Ncat \ar[r] \ar[drr] & \Zcat \ar[r] \ar[dr] & \Rcat \ar[d]\\
    & & & \Top}
\end{equation}
horizontal arrows are inclusions and pullbacks are given by restriction. Conversely, one may extend 
\begin{itemize}
\item $U:\ncat \to \Top$ to $U':\Ncat \to \Top$ by assigning $U'(k) = U(n)$ for $k>n$;
\item $U':\Ncat \to \Top$ to $U'':\Zcat \to \Top$ by assigning $U''(k) = \emptyset$ for $k<0$;
\item $U'':\Zcat \to \Top$ to $U''':\Rcat \to \Top$ by assigning $U'''(a) = U''\left(\lfloor a \rfloor\right)$ for all $a \in \R$
\end{itemize} 
(here $\lfloor \cdot \rfloor$ indicates the floor function). That such constructions are possible in each case depicted above is a pleasant consequence of the fact that $\Top$ admits left Kan extensions. 

Since we may pass from one of these functors to another, we will henceforth treat all filtered spaces as functors $\Rcat \to \Top$. Letting $\Vect$ denote the category of vector spaces and linear maps over a fixed underlying field, we note that any functor $H: \Top \to \Vect$ (such as singular homology) induces a push-forward from $\Top^\Rcat$ to $\Vect^\Rcat$ via post-composition. The resulting structure is a persistence module.
\begin{definition}
The category $\Mod$ of {\em persistence modules} is $\Vect^\Rcat$ -- its objects are functors $U:\Rcat \to \Vect$ and morphisms in $\Mod(U,V)$ are natural transformations $U \To V$.
\end{definition}
The morphisms from $U$ to $V$ in $\Mod$ admit a convenient pointwise description as collections of linear maps $\{\Phi(a):U(a) \to V(a) \mid a \in \R\}$ which satisfy the following property. Across all choices of $a \leq b$ in $\R$, the following diagram commutes:
\[
\xymatrixcolsep{0.6in}
\xymatrix{
U(a) \ar@{->}[r]^{U(a \leq b)} \ar@{->}[d]_{\Phi(a)} & U(b) \ar@{->}[d]^{\Phi(b)}  \\
V(a)  \ar@{->}[r]_{V(a \leq b)} & V(b) 
}
\]
It is often convenient to pass to a more general setting by considering different choices of target categories.  We therefore follow \cite{bs} and work with $\RC$ for an arbitrary category $\bC$, keeping in mind that this functor category specializes to $\Mod$ whenever $\bC = \Vect$. 
We call $\RC$ the category of \emph{persistence modules with values in $\C$}.

\subsection{The Interpolation Lemma}\label{ssec:interleave}

For each $e \geq 0$, one has a \emph{translation} functor $T_e: \Rcat \to \Rcat$ (sending $a$ to $a+e$) as well as a unique natural transformation $\sigma_e$ from the identity $1_\Rcat$ to this functor. It is readily confirmed that every such translation induces an endofunctor on $\RC$ which
\begin{enumerate}
\item sends each $U \in \RC\ob$ to $UT_e$ satisfying $UT_e(a) = U(a+e)$ for $a \in \R$, and
\item admits a distinguished natural transformation
  from the identity.
\end{enumerate}

\begin{definition}\label{def:int}
Given $e \geq 0$, two functors $\U,\V \in \RC\ob$ are said to be \emph{$e$-interleaved} if there are morphisms $\Phi: \U \to \V T_e$ and $\Psi: \V \to \U T_e$ in $\RC$ satisfying $(\Psi T_e)\Phi = \U\sigma_{2e}$ and $(\Phi T_e)\Psi = \V\sigma_{2e}$, as encoded in commutativity of the following diagrams:
\[
\xymatrix{
U \ar@{->}[rr]^{U\sigma_{2e}} \ar@{->}[dr]_{\Phi} & & UT_{2e} & &  UT_e \ar@{->}[dr]^{\Phi T_e} & \\
& VT_e \ar@{->}[ur]_{\Psi T_e} &   &   V \ar@{->}[rr]_{V\sigma_{2e}} \ar@{->}[ur]^{\Psi}  & & VT_{2e}                                                                           
}
\]
\end{definition}

The {\em interleaving distance} \cite{bs,ccsgo, cdsgo, les} on $\RC\ob$ is defined as follows:
\[
\intdist(\U,\V) = \inf \left\{ e \geq 0 \ | \ U, V \text{ are $e$-interleaved}\right\},
\]
with the understanding that $\intdist(\U,\V)=\infty$ if no interleaving exists. 

We want to say that $\RC\ob$ together with the interleaving distance is a metric space. To make this possible,
throughout this article, we relax the usual requirements for a metric space $(M,d)$ in three ways: 
\begin{enumerate}
\item we allow $d(x,y)$ to attain the value $+\infty$, 
\item we allow $d(x,y) = 0$ for $x \neq y$ in $M$, and
\item we allow $M$ to be a class rather than a set. 
\end{enumerate}
In other words, we work with {\bf symmetric Lawvere metric spaces} as defined in \cite{lawvere:metric}.

At times, we will not allow the third generalization. That is, we will need the collection of elements in a metric space to be a set. Let $\Met$ be the category whose objects are metric spaces with a set of elements, and whose morphisms are {\em non-expansive} or {\em 1-Lipschitz} maps\footnote{That is, a map $f:(M,d_M) \to (N,d_N)$ satisfying $d_N(f(x),f(y)) \leq d_M(x,y)$ for all $x,y \in M$.}.
Similarly, let $\Cat$ denote the category of small categories and functors.

\begin{theorem}[\cite{bs}] \label{thm:cat-met}
For each category $\bC$ and functor $H:\bC \to \bD$,
  \begin{enumerate}
    \item The pair $\left(\RC\ob,\intdist\right)$ is a metric space.
    \item The map $H^\Rcat:\RC\ob \to \RD\ob$ sending $\U$ to $H\U$ is 1-Lipschitz with respect to $\intdist$.
    \end{enumerate}
	Specializing to small categories, these assignments define a functor $\bullet^\Rcat: \Cat \to \Met$.
\end{theorem}
We call $\left(\RC\ob,\intdist\right)$ the \emph{metric space of persistence modules with values in $\C$} and call the functor $\bullet^{\Rcat}$ the \emph{persistence module functor}.

Recall that a metric space $(M,d)$ is a {\em path metric space} if for each $x,y$ in $M$ the infimum of the lengths of all paths between them equals $d(x,y)$. The following {\bf interpolation lemma} establishes that $\Mod\ob$ is a path metric space when endowed with the interleaving distance as a metric.
\begin{lemma}[\cite{ccsgo}] \label{lem:interpolation}
Given $e \geq 0$ and two $e$-interleaved persistence modules $\U_0$ and $\U_e$ in $\Mod\ob$, there exists a one-parameter family $\set{\U_a \mid a \in (0,e)}$ in $\Mod\ob$ so that $\U_a$ and $\U_b$ are  $|a-b|$-interleaved for all $a,b \in [0,e]$.
\end{lemma}
Note in general that the interpolating family $U_t$ is not unique, and that this lemma need not hold for general categories $\RC$. (We will describe additional hypotheses on $\bC$ under which the interpolation lemma holds for $\RC$ in Proposition \ref{prop:kaaan}.)

As mentioned in the Introduction, our main result is a higher interpolation lemma. A simple example illustrating that higher interpolations are not always possible may be found in \cite{dsn}, and we will reproduce it here in Section \ref{ssec:failure}. On the other hand, a pathway towards the desired generalization is provided by the following {\bf sharp interpolation lemma}: not only can one find an interpolating family of modules, but one can also find a compatible family of interleaving maps between them.
\begin{lemma}[\cite{cdsgo}] \label{lem:stronger}
Given persistence modules $\U_0$ and $\U_e$ along with morphisms $\Phi:\U_0 \to \U_eT_e$ and $\Psi:\U_e \to \U_0T_e$ which realize an $e$-interleaving, there exist
\begin{enumerate}
\item persistence modules $\set{\U_a \mid a \in (0,e)}$, and
\item module morphisms $\Phi_a^b:\U_a \to \U_bT_{b-a}$ and $\Psi_b^a:\U_b \to \U_aT_{b-a}$ for all $a \leq b$ in $[0,e]$,
\end{enumerate}
so that 
\begin{enumerate}
\setcounter{enumi}{2}
\item $\Phi_0^e = \Phi$ and $\Psi_e^0 = \Psi$,
\item $\Phi_a^b$ and $\Psi_b^a$ realize a $(b-a)$-interleaving between $\U_a$ and $\U_b$, and
\item \label{item:coherent} $(\Phi_b^c T_{b-a}) \Phi_a^b = \Phi_a^c$ and $(\Psi_c^b T_{b-a}) \Psi_b^a = \Psi_c^a$ hold for all $a \leq b \leq c$ in $[0,e]$.
\end{enumerate}
\end{lemma}
In general, the intermediate modules $U_a$ and maps $\Phi_a^b$ and $\Psi_b^a$ are not uniquely defined. 

\subsection{Failure of Higher Interpolation} \label{ssec:failure}

\begin{figure}[h!]
\includegraphics[scale=0.6]{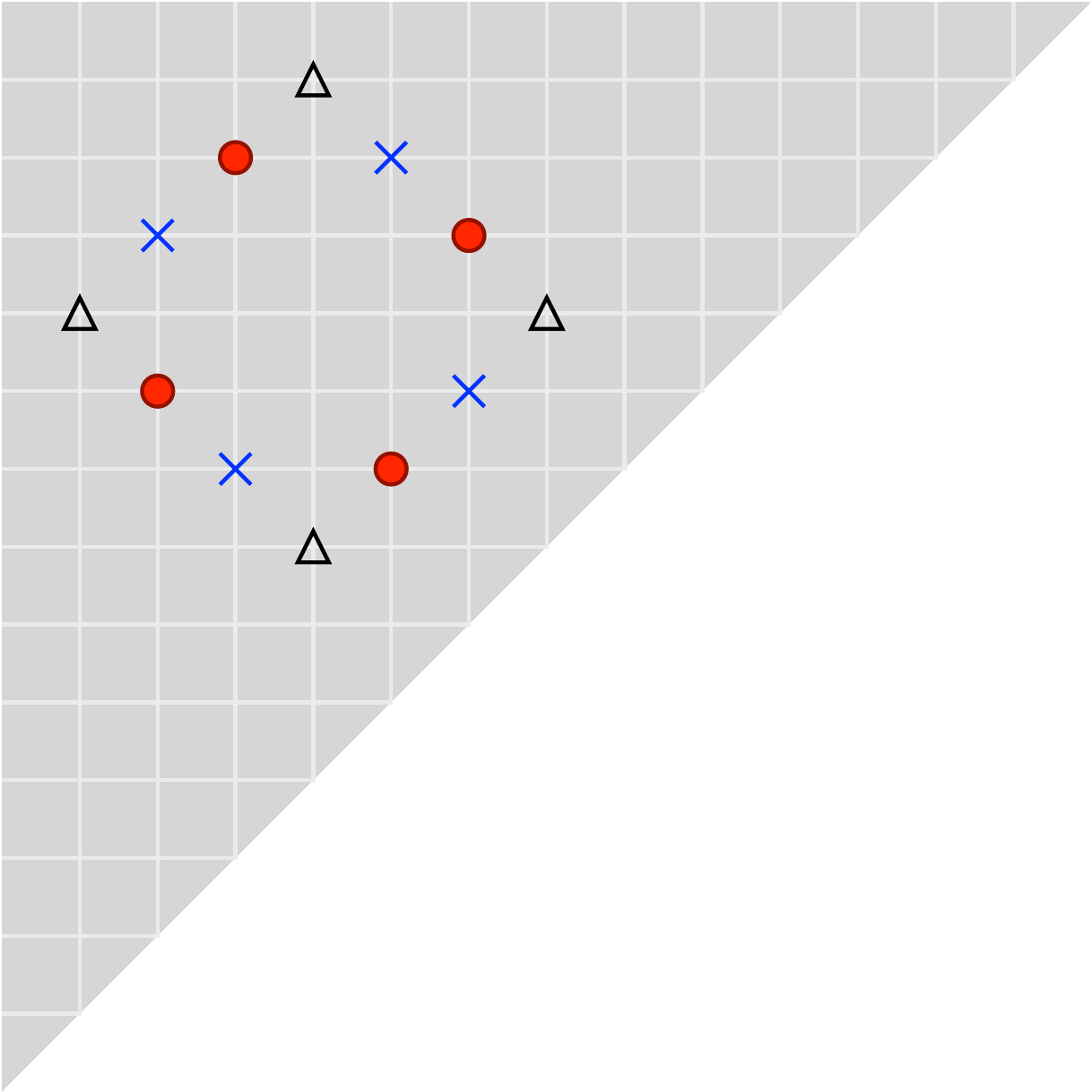}
\caption{Three overlaid persistence diagrams $\Delta, \times$ and $\bullet$. The corresponding persistence modules have pairwise interleaving distance $1$, but there is no persistence module within distance $e$ of all three for any $e < 1$.}
\label{fig:incoh}
\end{figure}

The main result of \cite{cb} asserts that (isomorphism classes of) {\em tame}\footnote{These are modules $U:\Rcat \to \Vect$ for which $U(t)$ is finite-dimensional for each $t$.} persistence modules are faithfully represented by their {\em persistence diagrams} \cite{zc,cseh,ccsgo}, which are multi-sets of points in the upper half-plane\footnote{To be precise, one needs to take decorated persistence diagrams~\cite{cdsgo}. This result extends to the {\em q-tame} modules described in~\cite{cdsgo,ccbds}.}. Moreover, these diagrams admit a {\em bottleneck distance} $\botdist$ and the following {\bf isometry theorem} establishes that the assignment $\dgm$ which sends a (tame) persistence module to its corresponding diagram preserves distances.
\begin{theorem}[\cite{cdsgo, bs, les}]
The equality
\[
\intdist (\U,\V) = \botdist (\dgm~\U, \dgm~\V)
\]
holds across all pairs $\U, \V$ of tame persistence modules\footnote{This result also holds for q-tame persistence modules~\cite{cdsgo}.}.
\end{theorem}

It was shown in \cite{dsn} that higher interpolations may fail to exist even for simple choices of $A \hookrightarrow M$. 

\begin{example} \label{ex:higher-order}
Let $(A,d)$ be the three-point metric space $\{x_1,x_2,x_3\}$ with $d(x_i,x_j) = 1$ for $i\neq j$, and let $M$ be $A$ together with $x_0$ where $d(x_0,x_i) = \frac{1}{2}$ for $i \geq 1$. Let $f:A \to \Mod\ob$ be the function whose images $f(x_j)$ prescribe the three persistence diagrams shown in Figure~\ref{fig:incoh}. Growing a ball of radius one (in the $L_\infty$ norm) around points in any one  diagram subsumes points in the other two diagrams, whence the pairwise bottleneck distances satisfy
\[
\botdist(\dgm~f(x_i), \dgm~f(x_j)) = 1 \text{ for } i \neq j.
\]  Note that any one of these diagrams lies at distance exactly $1$ from the other two. However, no persistence diagram is within distance $< 1$ of each of the three persistence diagrams in Figure~\ref{fig:incoh}. Thus, it follows from the isometry theorem that there is no tame persistence module which may be assigned to $x_0$ in any extension of $f$ without strictly increasing the Lipschitz constant.
\end{example}
This sharp interpolation lemma suggests the reason for the failure of the higher-order interpolation in Example~\ref{ex:higher-order}: the $1$-interleavings do not satisfy the compatibility condition \eqref{item:coherent}. 
 
\section{Categories from Metric Spaces}\label{sec:met-cat}

In Theorem \ref{thm:cat-met} we defined the persistence module functor $\bullet^\Rcat : \Cat \to \Met$. The central goal of this section is to describe the construction of a functor $\Met \to \Cat$, whose interaction with $\bullet^\Rcat$ will be of crucial importance in the proof of our main result. To this end, consider $(M,d) \in \Met\ob$ and let $M\R$ denote the product $M \times \R$ equipped with the binary relation
\[
(x,s) \leq_M (y,t) \text{ if and only if }d(x,y) \leq t-s.
\]
Henceforth, we will often drop the subscript and simply write $(x,s) \leq (y,t)$, relying on context for clarity. 
We call $M\R$ the \emph{spacetime} of $M$.
The following result is straightforward.
\begin{proposition}
The relation $\leq$ induces a pre-order on $M\R$. Moreover, if $d$ is a genuine metric in the sense that $d(x,y) = 0$ holds only for $x=y$, then $\leq$ induces a partial order on $M\R$.
\end{proposition}
\begin{proof}
Since $d(x,x) = 0$, we have reflexivity: $(x,s) \leq (x,s)$. Turning to transitivity, assume $(x,s) \leq (y,t)$ and $(y,t) \leq (z,u)$. By the triangle inequality, we have 
\[
d(x,z) \leq d(x,y) + d(y,z) \leq (t-s) + (u-t) = u-s.
\] Hence, $(x,s) \leq (z,u)$ so $\leq$ is transitive. Finally, if $(x,s) \leq (y,t) \leq (x,s)$, then $0 \leq d(x,y) \leq \min(s-t,t-s)$. So $s=t$. Thus, we have anti-symmetry only if $d(x,y) = 0$ forces $x=y$ in $M$. 
\end{proof}

\begin{figure}
\includegraphics[scale=0.35]{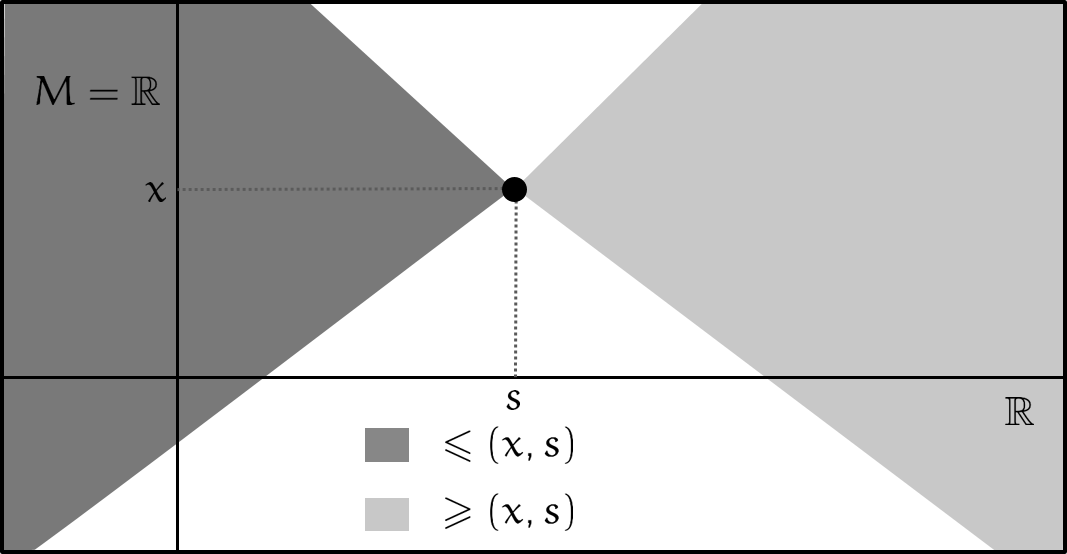}
\caption{An illustration of the order $\leq$ on $M\R$ in the special case where $M = \R$ with the usual metric. Given $(x,s) \in M\R$, the lighter region consists of the up-set $(y,t) \geq (x,s)$ while the darker region consists of the down-set. The up-set is inside the future light cone and the down-set is inside the past light cone.}
\end{figure}

Since $M\R$ is a preordered set, it may be treated as a {\em thin} category\footnote{A category is thin if it admits at most one morphism between any pair of objects.}. Given a $1$-Lipschitz map $f \in \Met(M,N)$, define $f\R: M\R \to N\R$ via the mapping $(x,s) \mapsto (f(x),s)$. 
\begin{theorem} \label{thm:met-cat}
  The assignment $\bullet\R$ prescribes a functor $\Met \to \Cat$, which we call the \emph{spacetime functor}.
\end{theorem}
\begin{proof}
  We first confirm that $f\R: M\R \to N\R$ is a morphism in $\Cat$ whenever $f:(M,d_M) \to (N,d_N)$ is $1$-Lipschitz. If $(x,s) \leq (y,t)$, we have $d_M(x,y) \leq t-s$. Since $f$ is $1$-Lipschitz, we obtain $d_N(f(x),f(y)) \leq d_M(x,y) \leq t-s$. So by definition, 
	\[
	f\R(x,s)=(f(x),s) \leq_N (f(y),t) = f\R(y,t).
	\] Thus, $f\R$ is order-preserving. It is easy to confirm that $1_M\R = 1_{M\R}$, so we turn to the task of establishing functoriality. Consider $M \xto{f} N \xto{g} P$ in $\Met$, and note that
\[
g\R \circ f\R(x,s) = g\R(f(x),s) =(gf(x),s)=[(gf)\R](x,s),
\]
which concludes the proof.
\end{proof}

With the existence of the spacetime functor $\bullet\R:\Met \to \Cat$ established, one might hope for an adjunction with the persistence module functor $\bullet^\Rcat:\Cat \to \Met$ from Theorem \ref{thm:cat-met}. If such an adjunction existed, then for each metric space $M \in \Met\ob$ and category $\bC \in \Cat\ob$, we would expect a natural bijection of sets between
\[
\Cat(M\R,\bC) \text{ and } \Met(M,\RC),
\] 
i.e., the set of functors from $M\R$ to $\bC$ would correspond with $1$-Lipschitz maps from $M$ to $\RC$. Example \ref{ex:higher-order} confirms that there is no such bijection in general, whence our functors $\bullet^\Rcat$ and $\bullet\R$ do not constitute an adjoint pair. Instead, we seek solace in the existence of a {\em unit}, as described below. 

Let $F$ denote the endofunctor on $\Met$ arising from the following composition:
\begin{align}\label{eq:F}
F: \Met \xto{\bullet\R} \Cat \xto{\bullet^\Rcat} \Met.
\end{align}
Chasing definitions, one can explicitly describe the effect of $F$ on the objects and morphisms of $\Met$: each metric space $M$ is mapped to $M\R^\Rcat$ (with the interleaving distance), and every $1$-Lipschitz $f:M \to N$ is sent to the map $M\R^\Rcat \to N\R^\Rcat$ which takes $U:\Rcat \to M\R$ to $f\R\circ U:\Rcat \to N\R$.
Note that the objects of $F(M)$ consist of \emph{world lines} in the spacetime of $M$.
\begin{theorem}
\label{thm:unit}
  The functor $F$ admits a natural transformation $\eta: \id_{\Met} \To F$ from the identity endofunctor on $\Met$. Furthermore, for each metric space $M$, $\eta_M$ is the isometric
embedding of $M$ into the metric space of persistence modules valued in $M\R$ given by the constant world lines in the spacetime $M\R$.
\end{theorem}
\begin{proof}
  For each $(M,d_M) \in \Met\ob$, we require a 1-Lipschitz map $\eta_M$ in $\Met(M,FM)$ which sends points of $M$ to functors $\Rcat \to M\R$. We provisionally define this map as follows: for each $x \in M$, let $\eta_M(x)$ be the functor which sends $s \in \R$ to $(x,s)$ and $s\leq t$ to $(x,s) \leq (x,t)$. To check that the latter inequality holds in $M\R$, we verify that $d_M(x,x) \leq t-s$. This definition sends identities to identities and respects composition since $M\R$ is a thin category. Thus, $\eta_M(x)$ is indeed a functor.

 Next, we confirm that $\eta_M$ is $1$-Lipschitz. Letting $x$ and $y$ be points in $M$ with $e = d_M(x,y)$, it suffices to construct an $e$-interleaving between $\eta_M(x)$ and $\eta_M(y)$. By definition, $(x,s) \leq (y,t)$ if and only if $d_M(x,y) \leq t-s$. Thus, we have the inequalities
\begin{equation} \label{eq:ineq1}
(x,s) \leq (y,s+e) \text{ for all } s \in \R,
\end{equation}
whose images under $\bullet^\bR$ yield a morphism $\Phi:\eta_M(x) \to \eta_M(y)T_e$ in $FM = M\R^\bR$. And similarly, we have the inequalities
\begin{equation} \label{eq:ineq2}
(y,t) \leq (x,t+e) \text{ for all } t \in \R,
\end{equation}
whose images under $\bullet^\bR$ assemble into a morphism $\Psi:\eta_M(y) \to \eta_M(x)T_e$ in $FM$. One can readily check that $\Phi$ and $\Psi$ furnish the desired $e$-interleaving of $\eta_M(x)$ and $\eta_M(y)$: since $M\R$ is a thin category, the diagrams from Definition \ref{def:int} must commute. Thus, $\eta_M:M \to FM$ is $1$-Lipschitz as desired.

Note that because  $(x,s) \leq (y,t)$ if and only if $d_M(x,y) \leq t-s$, there is no smaller value than $e$ for which that inequalities \eqref{eq:ineq1} and \eqref{eq:ineq2} hold. Thus $\eta_M(x)$ and $\eta_M(y)$ are not $e'$-interleaved for any $e'<e$. Therefore the interleaving distance between $\eta_M(x)$ and $\eta_M(y)$ is $e$, and $\eta_M$ is in fact an isometric embedding.

  Finally, we check that the assignment $x \mapsto \eta_M(x)$ prescribes a natural transformation. Given $f:(M,d_M) \to (N,d_N)$ in $\Met$, we must verify that the following diagram commutes:
  \[
		\xymatrixcolsep{0.4in}
		\xymatrixrowsep{0.4in}
    \xymatrix{
      M \ar[r]^{\eta_M} \ar[d]_f & FM \ar[d]^{F f} \\
      N \ar[r]_{\eta_N} & FN
      }
  \]
  Pick any $x\in M$. For all $s \in \R$, we have 
	\begin{align*}
\big([F f \circ \eta_M](x)\big)(s) &= Ff (x,s) \\ 
                                                   & = (f(x),s) = \big([\eta_N \circ f](x)\big)(s) 
	\end{align*}
	so our diagram commutes and $\eta$ is a natural transformation.
\end{proof}

\section{Coherence and Higher Interpolation} \label{sec:coh}

Throughout this section, we fix a choice of category $\C$ and metric space $A \in \Met$. We also let $\eta:\id_{\Met} \To F$ be the natural transformation from the proof of Theorem \ref{thm:unit}. 
For a functor $G:A\R \to \C$ define $\theta(G) = G^\Rcat \circ \eta_A$:
\begin{equation} \label{eq:coherent}
A \xto{\eta_A} FA = A\R^\Rcat \xto{G^\Rcat} \RC.
\end{equation}

\begin{definition}
\label{def:coh}
The 1-Lipschitz functions $g:A \to \RC\ob$ which lie in the image of $\theta$ are called {\bf coherent}. In the special case where $\C = \Vect$, such functions are called {\bf coherent persistence modules}.
\end{definition}
By definition, for every coherent 1-Lipschitz map $g:A \to \RC$ there is some functor $G:A\R \to \bC$ satisfying $g = \theta(G)$. The map $G^\Rcat$ now serves as an extension of $g$ across $\eta_A: A \to A\R^\Rcat$ because the following diagram commutes by the definition of $\theta$:
\begin{equation} \label{cd:extension-thickening}
  \xymatrix{
    A \ar[r]^{g} \ar[d]_{\eta_A} & \RC \\
    A\R^{\Rcat} \ar@{-->}[ur]_{{G}^{\Rcat}}
}
\end{equation}

If $A$ isometrically embeds into a larger metric space $M \in \Met$, then it is easy to check that $A\R$ is a full subcategory of $M\R$. Given any functor $G: A\R \to \bC$, one has the following functor extension problem:
\begin{equation}
\label{eq:funcext}
  \xymatrix{
    A\R \ar[r]^{G} \ar@{_{(}->}[d] & \C \\
    M\R \ar@{-->}[ur]_{\hat{G}}
}
\end{equation}

Recall that the category $\C$ is (co)complete if it has all (co)limits. The solution to problems such as (\ref{eq:funcext}) for functors taking values in (co)complete categories is furnished by {\bf Kan extensions} \cite[Ch X]{maclane}.
\begin{proposition} \label{prop:kaaan}
An extension $\hat{G}$ of $G$ exists under any of the following circumstances: 
\begin{itemize}
\item  if $\C$ is cocomplete, we can take $\hat{G}$ to be the {\bf left} Kan extension $\mathrm{Lan }G$ of $G$,
\item  if $\C$ is complete, we can take $\hat{G}$ to be the {\bf right} Kan extension $\mathrm{Ran }G$ of $G$,
\item  if $\C$ is bicomplete and abelian, we can take $\hat{G}$ to be the {\bf image} of the universal natural transformation $\mathrm{Lan }G \To \mathrm{Ran }G$.
\end{itemize}
\end{proposition}
If $\C = \Vect$ (as in the case of persistence modules) then we have all three extensions, but if $\C = \Top$ (as in the case of filtered topological spaces), then we only have the left and right extensions.  The following theorem is the main result of this paper.
\begin{theorem}
\label{thm:main}
  Let $A$ be the subspace of a metric space $M \in \Met$, and assume that the 1-Lipschitz map $g:A \to \RC\ob$ is coherent. If $\C$ is (co)complete, then $g$ admits a coherent 1-Lipschitz extension $\hat{g}:M \to \RC\ob$.
\end{theorem}
\begin{proof}
  Since $g$ is coherent, it equals $\theta(G)$ for some functor $G:A\R \to \C$. By Proposition \ref{prop:kaaan} and the (co)completeness hypothesis on $\C$, there is an extension $\hat{G}:M\R \to \C$ of $G$ as in (\ref{eq:funcext}). Note that the following diagram of metric spaces and 1-Lipschitz maps commutes:
  \begin{equation*}
    \xymatrix{
      A \ar@{_{(}->}[d] \ar[r]^{\eta_A} & A\R^{\Rcat} \ar@{_{(}->}[d] \ar[r]^{G^{\Rcat}} & \RC \\
      M \ar[r]_{\eta_M} & M\R^{\Rcat} \ar[ur]_{\hat{G}^{\Rcat}} &  
}
  \end{equation*}
The square on the left commutes because $\eta$ is a natural transformation. 
The triangle on the right commutes since $G^{\Rcat}(F) = G \circ F = \hat{G} \circ i \circ F = \hat{G}^{\Rcat}(i(F))$, where $i: A\R \hookrightarrow M\R$ and the middle equality is by~\eqref{eq:funcext}. Since the composite in the top row of our diagram equals $g$, it is immediately seen that the desired extension $\hat{g}:M \to \RC$ is given by $\theta(\hat{G}) = \hat{G}^\Rcat \circ \eta_M$.
\end{proof}

The $\C = \Vect$ specialization of Theorem \ref{thm:main} yields the higher interpolation lemma promised in the Introduction. In this case, for a given $G$ satisfying $\theta(G) = g$ we have at least three possible choices\footnote{For explicit calculations and a comparison of all three extensions in the context of the sharp interpolation lemma, consult \cite[Sec 3.5]{cdsgo} (and particularly Prop 3.6 therein). The image extension is optimal among the three in the sense that it satisfies two universal properties instead of one.} of $\hat{G}$ arising from Proposition \ref{prop:kaaan}. Regardless of which $\hat{G}$ is chosen, the map $\theta(\hat{G})$ is itself coherent by construction, and hence admits further extensions to larger metric spaces.


\section{Consequences} \label{sec:cons}

In this section we describe some applications of Theorem \ref{thm:main}.  

\subsection{Discrete and Continuous Interpolation} \label{ssec:discon}

Let $U_1,\ldots,U_n$ be a collection of $n \geq 1$ persistence modules and let $e \geq 0$ be a fixed constant. Assume further that $U_i$ and $U_j$ are $2e$-interleaved for all $i \neq j$. Let $A = \{a_1,\ldots,a_n\}$ be the metric space where all nontrivial distances $d(a_i,a_j)$ equal $2e$, and note that we may describe each $U_i$ as the image $g(a_i)$ of a 1-Lipschitz map $g: A \to \Mod\ob$. Recall the translation functor $T$ and the natural transformation $\sigma$ as defined in Section \ref{ssec:interleave}. The following result provides an easily-computable criterion for coherence (compare with Lemma \ref{lem:stronger} as well as \cite[Thm 4.2]{dsn}).
\begin{proposition}
\label{prop:oldcoh}
The map $g:A \to \Mod\ob$ is coherent if and only if for all distinct $i,j$ in $\{1 \ldots n\}$ there exist morphisms $\Phi_{ij}:U_i \to U_jT_{2e}$ in $\Mod$ satisfying:
\begin{enumerate} 
\item $(\Phi_{ji}T_{2e}) \circ \Phi_{ij} = U_i\sigma_{4e}$ for all distinct $i,j$ and
\item $(\Phi_{jk} T_{2e}) \circ \Phi_{ij} = \Phi_{ik}T_{2e}$ for all distinct $i,j,k$.
\end{enumerate}
\end{proposition}
\begin{proof} 
Assume first that $g$ is coherent, so there exists a functor $G:A\R \to \Vect$ satisfying $g = G^\Rcat \circ \eta_A$. To define the desired morphisms $\Phi_{ij}:U_i \to U_jT_{2e}$ in $\Mod$, it suffices to construct linear maps $\Phi_{ij}(s):U_i(s) \to U_j(s+2e)$ indexed by $s \in \R$ and $i,j \leq n$ (subject to the constraint that diagrams from Definition \ref{def:int} commute). To this end, note that $(a_i,s) \leq (a_j,s+2e)$ in $A\R$ because $d(a_i,a_j) \leq 2e$ in $A$ by assumption. Define $\Phi_{ij}(s)$ to be the image under $G$ of $(a_i,s) \leq (a_j,s+2e)$. Since $A\R$ is a thin category and $G$ is a functor, all the required diagrams commute and the two desired properties follow. 

On the other hand, assume the existence of maps $\Phi_{ij}$ satisfying the two properties from the statement of this proposition. We will use them to define a functor $G:A\R \to \Vect$ which renders $g$ coherent by satisfying $\theta(G) = g$.  Set $G(a_i,t) = U_i(t)$ for all $i \leq n$ and $t \in \R$. Given $(a_i,s) \leq (a_j,t)$ in $A\R$, we either have $t-s \geq 2e = d(a_i,a_j)$ if $i \neq j$ or simply $t - s \geq 0$ if $i=j$; so, define 
\begin{align*}
G\left((a_i,s) \leq (a_j,t)\right) = \begin{cases}
																						\Phi_{ij}(s) \circ T_{(t-s)-2e} & i \neq j \\
																						U_i \sigma_{t-s} & i = j
																			\end{cases}.
\end{align*} 
Straightforward calculations confirm that $G$ is a functor, and that $\theta(G) = g$.
\end{proof}

Let $M$ be the metric space which consists of $A$ above, along with an additional point  $a$ so that $d(a,a_i) = e$. The {\bf discrete interpolation problem} for persistence modules seeks to extend our 1-Lipschitz map $g:A \to \Mod\ob$ to a 1-Lipschitz map $\hat{g}:M \to \Mod\ob$ across the obvious inclusion $A \hookrightarrow M$. On the other hand, the {\bf continuous interpolation problem} for persistence modules seeks an extension of $g$ across the inclusion of $A$ as vertices of the standard $(n-1)$-simplex $\Sigma \subset \R^n$, given by:
\[
\Sigma = \left\{(x_1,\ldots,x_n) \in \R^n \mid x_1 + \cdots + x_n  = \sqrt{2}e \text{ and } x_j \geq 0\right\}
\] 

It follows immediately from Theorem \ref{thm:main} that the discrete and continuous interpolation problems both admit solutions (in triplicate) whenever the modules $U_1,\ldots,U_n$ are connected by morphisms $\Phi_{ij}$ which satisfy the two properties from Proposition \ref{prop:oldcoh}.
Note that in the latter case, $g$ extends not only to $\Sigma$ but to $\R^n$.

\subsection{\v{C}ech and Rips Complexes of Persistence Modules}

Let $(M,d_M)$ be an ambient metric space with a distinguished subspace $A \subset M$. We recall the {\bf Vietoris-Rips} complex $\Rips$ and the {\bf \v{C}ech} complex $\Cech$ on $A$. Both are abstract simplicial complexes filtered by a single non-negative real parameter; their vertices are the points of $A$, but the construction of higher-dimensional simplices sets them apart. In particular, a collection $[a_0,\ldots,a_n]$ of points in $A$ forms an $n$-simplex
\begin{enumerate}
\item in $\Rips(A,e)$ if and only if $d_M(a_i,a_j) \leq e$ for $0 \leq i < j \leq n$, and
\item in $\Cech(A,e)$ if and only if some $b$ in $M$ satisfies $d_M(a_i,b) \leq e$ for $0\leq i \leq n$.
\end{enumerate}
Any $b$ which is within distance $e$ of all the points $[a_0,\ldots,a_n]$ is called an {\em $e$-witness} for those points. It is well-known and immediate from the definitions that for each $e \geq 0$ one always has the following simplicial sandwich:
\[
\Cech(A,e) \hookrightarrow \Rips(A,2e) \hookrightarrow \Cech(A,2e),
\]
where the first inclusion follows  directly from the triangle inequality, and the second follows from the fact that any $a_i$ serves as a $2e$-witness for a simplex $[a_0,\ldots,a_n]$ in $\Rips(A,2e)$.

We may ask if these inclusions are tight. For the first inclusion, assume that $M$ is a path metric space (e.g., $\Mod\ob$), and that there exist $a_0,a_1\in A$ with $d_M(a_0,a_1)=2e+\delta$ for some $\delta>0$.  Then $[a_0,a_1]$ is not a $1$-simplex in $\Rips(A,2e)$ and there exists a path from $a_0$ to $a_1$ with length at most $2e+2\delta$. The midpoint of this path is a $(e+\delta)$-witness for $[a_0,a_1]$, which is therefore a $1$-simplex in $\Cech(A,e+\delta)$.
Thus $\Cech(A,e+\delta) \not\hookrightarrow \Rips(A,2e)$ for $\delta>0$.
The second inclusion $\Rips(A,2e) \hookrightarrow \Cech(A,2e)$ might be improved, depending on $M$. For example, $\Rips(A,2e)$ always includes into $\Cech(A,\nicefrac{2e}{\sqrt{3}})$ whenever $A$ is a subset of $\R^2$ with the standard Euclidean metric. 

When working within $\Mod\ob$, one typically strengthens the requirements in the definitions of Rips and \v{C}ech complexes slightly since the infimum over interleavings may not actually be attained. In particular, a collection $[\U_0,\ldots,\U_n]$ of persistence modules forms an $n$-simplex
\begin{enumerate}
\item in $\Rips(\Mod\ob,e)$ iff $\U_i$ and $\U_j$ are $e$-interleaved for all $0 \leq i < j \leq n$, and
\item in $\Cech(\Mod\ob,e)$ iff some $V$ is $e$-interleaved with $\U_i$ for all $0 \leq i \leq n$.
\end{enumerate}

It is straightforward to check that $\Rips(\Mod\ob,2e)$ does not include into $\Cech(\Mod\ob,d)$ for any $d < 2e$ by appealing to Figure \ref{fig:incoh} and the isometry theorem. On the other hand, the following result from \cite{dsn} characterizes those simplices of $\Rips(\Mod\ob,2e)$ which do include into $\Cech(\Mod\ob,e)$. 
\begin{theorem} \label{thm:rips-cech}
\label{thm:dsn}
Let $\U_0,\ldots,\U_n$ be a collection of persistence modules and let $e \geq 0$. Then, $[\U_0,\ldots,\U_n]$ is an $n$-simplex in $\Cech(\Mod\ob,e)$ if and only if there exist morphisms $\Phi_{ij}$ for $i \neq j$ which satisfy the conditions of Proposition \ref{prop:oldcoh}
\end{theorem}

Thus, a simplex in $\Rips(\Mod\ob,2e)$ forms a simplex in $\Cech(\Mod\ob,e)$ if and only if the module morphisms which realize the pairwise $2e$-interleavings can be chosen to commute (up to factors of the natural transformation $\sigma$). From our perspective here, the preceding result is a direct consequence of the discrete interpolation discussed in Section \ref{ssec:discon}.

\section*{Acknowledgments}

The authors are indebted to the anonymous referees for their thoughtful comments and suggestions. PB's work was partially supported by the AFOSR grant FA9550-13-1-0115. VN's work was supported by The Alan Turing Institute under the EPSRC grant number EP/N510129/1.

 \bibliographystyle{plain}
 \bibliography{pers}

\end{document}